\documentclass[a4paper,12pt]{amsart}
\usepackage[margin=2.5cm]{geometry}

\usepackage{amsmath}
\usepackage{amsfonts}
\usepackage{amssymb}
\usepackage{xcolor}
\usepackage{mathrsfs}
\usepackage{etoolbox}
\usepackage{array}
\AtBeginEnvironment{theorem}{\hangindent=2em}
\renewenvironment{proof}{{\bfseries Proof.}}{\qed}
\numberwithin{equation}{section} 
\newtheorem{theorem}{Theorem}[section] 
 
\newtheorem{cor}[theorem]{Corollary} 
\newtheorem{lemma}[theorem]{Lemma} 

\theoremstyle{definition}
 
\newtheorem{remark}[theorem]{Remark} 
\newtheorem{example}[theorem]{Example}

\newcommand{\n}{\newline}

\newcommand{\imp}{\Rightarrow}
\newcommand{\z}{\mathbb{Z}}

\newcommand{\m}{\gamma^{m}_{p}}

\newcommand{\thmref}[1]{Theorem~\ref{#1}}
\newcommand{\lemref}[1]{Lemma~\ref{#1}}

\newcommand{\corref}[1]{Corollary~\ref{#1}}
\newcommand{\eqnref}[1]{~{equation~(\ref{#1})}}

\numberwithin{equation}{section} 
 
\usepackage[backref]{hyperref}

\usepackage[font=small,labelfont=bf]{caption}

\begin{document}

\title[Combinatorial Structure of Inert and Ambiguous Classes]{Combinatorial Structure of Inert and Ambiguous Classes in Modular Group}
\author[D. Das]{Debattam Das}
\address{Indian Institute of Technology Kanpur, Kanpur 208016, Uttar Pradesh, India}
\email{debattam123@gmail.com }
\author[K. Gongopadhyay]{Krishnendu Gongopadhyay }
\address{Indian Institute of Science Education and Research (IISER) Mohali, Knowledge City,  Sector 81, S.A.S. Nagar 140306, Punjab, India}
\email{krishnendu@iisermohali.ac.in}
\author[K. Mishra]{Khushi Mishra}
\address{Indian Institute of Science Education and Research (IISER) Mohali, Knowledge City,  Sector 81, S.A.S. Nagar 140306, Punjab, India}
\email{mp22005@iisermohali.ac.in}
\makeatletter
\@namedef{subjclassname@2020}{\textup{2020} Mathematics Subject Classification}
\makeatother
\date{\today}
    \begin{abstract}
We study inert, and ambiguous conjugacy classes in the modular group $\mathrm{PSL}(2,\mathbb{Z})$ from a purely combinatorial perspective. Using word length in the free product representation $\mathbb{Z}_2 * \mathbb{Z}_3$ of the modular group, we obtain exact counting formulas and asymptotic growth rates for inert and ambiguous classes. Our results provide the first counting formulas for inert classes obtained independently of Sarnak’s analytic trace-based methods, while also establishing a combinatorial framework for ambiguous classes. 
\end{abstract}

\maketitle

\section{Introduction}
Recall that the modular group $\Gamma$ be generated by the transformations  
\[
T: z \mapsto z+1 \quad \text{and} \quad S: z \mapsto -\frac{1}{z}.
\]
The group $\Gamma$ can also be identified geometrically with $\mathrm{PSL}(2,\mathbb Z)$ and with the free product $\mathbb{Z}_2 \ast  \mathbb{Z}_3$.  This paper is inspired by the work of Sarnak \cite{sarnak}, where counting problems are discussed for reciprocal, ambiguous, and inert classes.

Let $\{\gamma\}_\Gamma$ denote the conjugacy class of an element $\gamma \in \Gamma$. The group $\Gamma$ acts on the hyperbolic plane $\mathbb{H}^2$, and its elements are classified as elliptic, parabolic, or hyperbolic according to the location of their fixed points, equivalently by the absolute value of the trace $t(\gamma)$; see \cite{MR1177168}. Primitive hyperbolic conjugacy classes in $\Gamma$ are in natural correspondence with primitive closed geodesics on the modular surface $\Gamma \backslash \mathbb{H}^2$.

In a seminal paper, Sarnak~\cite{sarnak} studied special conjugacy classes in $\Gamma$ using analytic methods arising from the Selberg trace formula and Gauss's quadratic reciprocity. He introduced three involutions that acts on $\Gamma$, namely $\phi_R$, $\phi_w$, and $\phi_{Amb}$, which together with identity $Id$ generate the Klein four-group
\[
K=\langle Id, \phi_R, \phi_w, \phi_{Amb} \rangle.
\]
These involutions are given by 
\[
\phi_R(\gamma) = \gamma^{-1}, \quad 
\phi_w(\gamma) = w^{-1} \gamma w, \quad 
\phi_{Amb} = \phi_R \circ \phi_w = \phi_w \circ \phi_R,
\]  
where  
\[
w = \begin{pmatrix} 1 & 0 \\ 0 & -1 \end{pmatrix} \in \mathrm{PGL}(2, \mathbb{Z}) \text{ is an reflection}.
\]  
Let $P$ denote the set of all primitive hyperbolic elements, and let $\Pi$ be the corresponding set of their conjugacy classes. Geometrically, each $p \in P$ corresponds to an oriented primitive closed geodesic on the quotient surface $\Gamma \backslash \mathbb{H}^2$, with length  
\[
\ell(p) = \log N(p), \quad \text{where } 
N(p) = \bigg(\frac{t(p) + \sqrt{t(p)^2 - 4}}{2}\bigg)^2.
\]

For any subgroup $H \subset K$, define  
\[
\Pi_H = \{\{p\} \in \Pi : \phi(\{p\}) = \{p\} \text{ for all } \phi \in H\}.
\]  
In particular, $\Pi_{\langle e \rangle} = \Pi$. Sarnak obtained asymptotic counting formulas for these symmetry classes with respect to hyperbolic length. The elements of $\Pi_H$ are referred to as \emph{reciprocal}, \emph{inert}, or \emph{ambiguous}, depending on whether $H$ is generated by $\phi_R$, $\phi_w$, or $\phi_{Amb}$.

The purpose of the present work is to study these classes from a purely combinatorial and symbolic perspective. Instead of hyperbolic length, we measure size by word length in the presentation $\Gamma \cong \mathbb{Z}_2 * \mathbb{Z}_3$. This approach allows explicit descriptions of inert and ambiguous conjugacy classes in terms of reduced words and leads to exact counting formulas by word length. For primitive and reciprocal classes, counting formulas based on word length were obtained by Basmajian and Suzzi Valli in \cite{combigrowth}. Here we consider the inert and ambiguous classes in the modular group, completing the combinatorial picture.

Representing hyperbolic conjugacy classes by cyclically reduced $(AB)$-words, we show that invariance under elements of $K$ translates into anti-periodicity or palindromic symmetry constraints on the associated exponent sequences. These constraints stratify admissible word lengths according to the highest power of $2$ dividing them, reflecting the involutive sign-reversing nature of the symmetry. This symbolic framework yields exact counting formulas for inert and ambiguous conjugacy classes of a given word length, distinguishes primitive from non-primitive classes, and provides precise asymptotic growth estimates. We summarize these structural results in Theorem~\ref{inert} and Theorem~\ref{ambiguous}. 

For related work on this direction, see \cite{paulin2} and \cite{paulin1}, where geometric methods based on divergent geodesics and common perpendiculars recover Sarnak’s counts for ambiguous and reciprocal classes in terms of geodesic lengths, equivalently traces. Further results building on Sarnak’s counting of reciprocal and primitive classes can be found in \cite{bv2},  \cite{bk}, \cite{es}, also see \cite{bmv}, \cite{das2024combinatorial}, \cite{dg2025}. 

Despite extensive research on counting special classes in combinatorial groups from a group-theoretic perspective, see, e.g., \cite{gs},  \cite{park}, \cite{rivin},  to the best of our knowledge, no results exist for inert classes obtained by methods other than Sarnak’s. This work provides the first such results using basic combinatorial methods from group theory (see, e.g., \cite{mgs}). Our approach not only complements existing analytic methods for inert and ambiguous classes in the modular group, but also offers a new combinatorial perspective on this  problem.
\subsection{Notations and Main Results}
We now introduce notation and state the main results.

Define $\langle P \rangle$ to be the set of all products of powers of primitive elements,
\[
\langle P \rangle := \left\{\, \prod_{i=1}^{n} p_i^{k_i} : n\geq 1,\ p_i \in P,\ k_i \in \mathbb{Z}\,\right\},
\]
and let
\[
\Pi_H' := \{\{p\}\in\langle\Pi\rangle : \phi(\{p\})=\{p\}\ \forall\,\phi\in H\},
\]
where $\langle\Pi\rangle$ denotes the set of conjugacy classes in $\langle P\rangle$. We write
\[
\mathscr{I} := \Pi_{\langle\phi_w\rangle}', 
\qquad
\mathscr{A} := \Pi_{\langle\phi_{\mathrm{Amb}}\rangle}'
\]
and refer to $\mathscr I$ and $\mathscr A$ as the sets of inert and ambiguous classes, respectively.

Using the identification $\Gamma \simeq \mathbb{Z}_2 \ast  \mathbb{Z}_3$, we define normal forms for conjugacy classes in $\mathscr{I}$ and $\mathscr{A}$ via combinatorial methods and study their growth rates by systematic counting of representatives.

Let the generators of $\mathbb{Z}_2 \ast  \mathbb{Z}_3$ be $A$ and $B$, with $A$ corresponding to $S: z \mapsto -1/z$ and $B$ corresponding to $ST: z \mapsto -1/(z+1)$. Henceforth we write $S=A$ and $ST=B$. For $\gamma \in \Gamma$, let $\|\gamma\|$ denote its word length with respect to the symmetric generating set $\{A,B,B^{-1}\}$, and let $W$ be the set of reduced words in these generators. Let $W_{2t}$ be the set of reduced words of length $2t$, and $W_{2t}(AB)$ the set of $(AB)$–words of length $2t$ in $W$.

The conjugacy class of $\gamma$ is denoted $[\gamma]$, with length
\[
\|[\gamma]\| = \min\{\|h\| : h \in [\gamma]\}.
\]
A word is cyclically reduced if every cyclic permutation remains reduced; such representatives realize the minimal length in their conjugacy class and differ only by cyclic permutation. An $(AB)$–word is a reduced word starting with $A$ and ending with $B$ or $B^{-1}$. As shown in \cite{combigrowth}, every $\gamma \in W$ not conjugate to a generator is conjugate to an $(AB)$–word. If an $(AB)$–word has even length $2t$, then a cyclic shift by $k$ positions is again an $(AB)$–word if and only if $k$ is even, and a $(BA)$–word if and only if $k$ is odd. 

\medskip Now, for real-valued functions $f, g : [0,\infty) \to \mathbb{R}$.  We write $f \sim g$ when
$$\frac{f(t)}{g(t)} \to 1 \qquad \text{as } t \to \infty.$$
We say $f\simeq g$, if there exist positive constants $C_1,C_2,C_3,C_4$ and
$t_0$ such that
$$
C_1 g(t)+C_3\leq f(t)\leq C_2 g(t)+C_4,~~~~~~\text{for all $t\geq t_0$.}
$$
 Also, whenever $f \simeq g$, we may equivalently denote this by $f(t)=\Theta(g(t))$.

\medskip We now have the following theorems. 
\begin{theorem}\label{inert}

\begin{enumerate}
    \item Let $\mathscr{I}_{\leq2t}$ be the set of all inert conjugacy classes of length at most $2t$. Then, as $t \to \infty$, we have
\[
\lvert\,\mathscr{I}_{\leq2t}\,\rvert \sim \frac{2^{\lfloor \frac{t}{2} \rfloor}}{\lfloor \frac{t}{2} \rfloor +1},
\]
where $\lfloor \frac{t}{2} \rfloor$ is the greatest integer less than or equal to $\frac{t}{2}$.
\item The set of all primitive inert classes of length at most $2t$ is asymptotic to the set of all inert conjugacy classes of length at most $2t$.
\end{enumerate}

\end{theorem}

\begin{theorem} \label{ambiguous}

\begin{enumerate}
    \item Let $\mathscr A_{\leq 2t}$ be the set of all ambiguous classes of length at most $2t$. Then we have
\[\lvert\,\mathscr A_{\le 2t}\,\rvert\simeq 2^{\frac{t}{2}}.\]
\item The set of all primitive ambiguous classes of length at most $2t$ is asymptotic to the set of all ambiguous classes of length at most $2t$.
\end{enumerate}
\end{theorem}

Using the asymptotic counts of inert and ambiguous classes together with the results for reciprocal classes from \cite{combigrowth} (see also \cite{dg2025}), we summarize the count in the following table.\\

\begin{table}[h]

\begin{tabular}{|m{7cm}|m{7cm}|}
\hline
\textbf{Conjugacy class sets} & \textbf{Cardinality of the conjugacy classes} \\
\hline
Reciprocal classes of length at most $2t$. 
& $\sim 2^{\lfloor t/2 \rfloor}$ \\
\hline
Inert conjugacy classes of length at most $2t$. 
& $\sim \dfrac{2^{\lfloor t/2 \rfloor}}{\lfloor t/2 \rfloor + 1}$ \\
\hline
Ambiguous conjugacy classes of length at most $2t$. 
& $\simeq 2^{t/2}$ \\
\hline
\end{tabular}
\caption{Growth of respective type of conjugacy classes}
\end{table}

In \textbf{Section 2}, we have studied the structures of the conjugacy classes of the inert elements in the modular group and deduced the asymptotic growth of the classes. Similarly we have studied the ambiguous classes and its asymptotic growth in  \textbf{Section 3}. The study of asymptotic growth of primitive classes of inert and ambiguous classes are also discussed in the respective sections.
\subsection*{Acknowledgment}
The authors thank Fr\'ed\'eric Paulin for his comments related to this work.

Part of this work was completed while Das was visiting the Indian Statistical Institute, Kolkata, and Gongopadhyay was visiting the Institut des Hautes \'Etudes Scientifiques, Paris-Saclay. The authors thank the respective institutes for hospitality and support during the course of this work.

Mishra gratefully acknowledges the partial support provided by the Department of Science and Technology, Government of India, under FIST Grant,  No. SR/FST/MS-I/2019/46,  for computational facilities.

\section{Combinatorial Structure of Inert Classes}  Any word in $W$ of length \(2t\) has the form
\[
p = AB^{\alpha_1} AB^{\alpha_2} \cdots AB^{\alpha_{t}},
\qquad \alpha_i\in\{\pm1\},
\]
where each block \(AB^{\alpha_i}\) contains exactly two letters.

If the conjugacy class \(\{p\}\) is fixed by \(\varphi_w\), then the sequence \((\alpha_1,\ldots,\alpha_{t})\) satisfies an anti-periodicity relation. In particular, there exists an integer \(m\mid t\) such that, after a suitable cyclic rotation, the word \(p\) becomes a repetition of the block
\[
(\alpha_1,\ldots,\alpha_m,-\alpha_1,\ldots,-\alpha_m,\dots)
\]
repeated \(t/m\) times.  Equivalently,
\[
\alpha_{m+i}=-\alpha_i \qquad \text{for all $i$ (indices modulo $2t$)},
\]
a relation that will be proved in the next lemma.

Thus the \(t\)-tuple \((\alpha_1,\ldots,\alpha_{t})\) is constructed from a fundamental
block of length \(2m\), which forces \(2t\) to be divisible by \(4\).  Hence we restrict our
analysis to inert classes with representatives of length \(4t\), denoted \(\mathscr{I}_{4t}\).

We now state the lemma establishing this symmetry.
\begin{lemma}\label{2.1}
Let $\{p\}$ be an inert class in $\mathscr{I}_{4t}$, and suppose
$\{p\}$ contains a cyclically reduced word of the form
\[
p \;=\; AB^{\alpha_1}AB^{\alpha_2}\cdots AB^{\alpha_{2t}}, 
\qquad \alpha_i \in \{\pm 1\}.
\]
If $\{p\}$ is fixed by $\phi_w$, then there exists an integer $m$ with 
$1 \le m \le t$ and $m \mid t$ such that, after a cyclic rotation of $p$, we have
\[
p \;=\;
\big(AB^{\alpha_1}\cdots AB^{\alpha_m}\big)
\big(AB^{-\alpha_1}\cdots AB^{-\alpha_m}\big)
\big(AB^{\alpha_1}\cdots AB^{\alpha_m}\big)
\big(AB^{-\alpha_1}\cdots AB^{-\alpha_m}\big)\cdots\cdots,
\]
that is,
\[
\alpha_{m+i} = -\alpha_i \qquad \text{for all $i$ (indices modulo $2t$)}.
\]
Equivalently, the $2t$--tuple $(\alpha_1,\dots,\alpha_{2t})$ has period $2m$
and the word $p$ is a $(t/m)$--fold repetition of a word of length $2m$.
\end{lemma}

\begin{proof}
Suppose $\{p\}\in \mathscr I_{4t}$ and let
$$
p \;\longleftrightarrow\;
\begin{pmatrix}
a & b\\
c & d
\end{pmatrix},
\qquad ad-bc=1,\quad |a+d|>2.
$$
Then
$$
\phi_w(p)=w^{-1}pw
\;\longleftrightarrow\;
\begin{pmatrix}
a & -b\\
-c & d
\end{pmatrix}.
$$

Moreover,
$$
\phi_w(A)
\;\longleftrightarrow\;
\phi_w\!\left(
\begin{pmatrix}
0 & -1\\
1 & 0
\end{pmatrix}\right)
=
\begin{pmatrix}
0 & 1\\
-1 & 0
\end{pmatrix}
\;\longleftrightarrow\; A,
$$
and
$$
\phi_w(B)
\;\longleftrightarrow\;
\phi_w\!\left(
\begin{pmatrix}
0 & -1\\
1 & 1
\end{pmatrix}\right)
=
\begin{pmatrix}
0 & 1\\
-1 & 1
\end{pmatrix}
\;\longleftrightarrow\; AB^{-1}A.
$$
Consequently,
$$
\phi_w(A)=A
\qquad\text{and}\qquad
\phi_w(B)=AB^{-1}A.
$$
Here $p \in\langle~ P ~\rangle$, since neither $A$ nor $B$ is hyperbolic, p cannot be conjugate to either of them. Moreover, every $\gamma \in W$ that is not conjugate to a generator is conjugate to an $(AB)$-word. Therefore, every conjugacy class of $\mathscr{I}_{4t}$ contains an AB-word.\\ Let $p$ itself be an $(AB)$-word of length 4t, that is $p$ = $AB^{\alpha_1}AB^{\alpha_2}....AB^{\alpha_{2t}}$, where $\alpha_i$'s are $\pm1$. $\phi_w$ is an inner automorphism, therefore $$\phi_w(AB^{\alpha_1}AB^{\alpha_2}\dots AB^{\alpha_{2t}}) = B^{-\alpha_1}AB^{-\alpha_2}\dots AB^{-\alpha_{2t}}A.$$
 Since $\{p\}\in \mathscr{I}_{4t}$, the conjugacy class $\{p\}$ is fixed by the automorphism 
$\phi_{w}$. This means that the words $p$ and $\varphi_{w}(p)$ must be conjugate in the 
group. The word $p$ has the standard form
\[
p = AB^{\alpha_1}\, AB^{\alpha_2}\cdots AB^{\alpha_{2t}},
\]
and applying $\phi_{w}$ sends each block $AB^{\alpha_i}$ to $B^{-\alpha_i}A$. Hence 
$\phi_{w}(p)$ begins with $B$ ends with $A$ and has the structure of a $(BA)$-word, while $p$ begins 
with $A$ ends with $B$ and is an $(AB)$-word. A $(BA)$-word can be conjugate to an $(AB)$-word only after 
an odd cyclic rotation: even rotations preserve the initial $B$, whereas odd rotations 
convert the leading $BA$ into $AB$. Therefore, any conjugating element must implement an odd 
cyclic shift of $\varphi_{w}(p)$.

Consider $L_m = B^{\alpha_m}A\, B^{\alpha_{m-1}}A \cdots B^{\alpha_1}A$. Conjugation by $L_m$ performs exactly the odd cyclic permutation 
required to convert the $(BA)$-pattern of $\varphi_{w}(p)$ back into the $(AB)$-pattern of 
$p$, giving the identity
\[
L_m\,\varphi_{w}(p)\,L_m^{-1} = p,
\]
which shows that $\varphi_{w}(p)$ and $p$ represent the same conjugacy class.

Moreover, writing out the equality $L_m\,\varphi_{w}(p)\,L_m^{-1} = p$ block-by-block yields a 
constraint on the exponents $\alpha_i$, i.e.,
$$ 
(B^{\alpha_m}A....B^{\alpha_2}AB^{\alpha_1})(B^{-\alpha_1}AB^{-\alpha_2}....AB^{-\alpha_{2t}}A)(B^{-\alpha_1}AB^{-\alpha_2}....AB^{-\alpha_m}) = AB^{\alpha_1}AB^{\alpha_2}....AB^{\alpha_{2t}}.
 $$   
 Then, $$\alpha_{m+1} = -\alpha_1,
 ~\alpha_{m+2} = -\alpha_2,
 \ldots, ~\alpha_{m+m} = -\alpha_m,~
 \alpha_{m+(m+1)} = -\alpha_{m+1},~
 \ldots, \alpha_{m} = -\alpha_t.$$
Matching corresponding $(AB)$-blocks on both sides 
forces the sign in the $(m+i)$-th position to be the negative of the sign in the $i$-th 
position. Hence
\[
\alpha_{m+i} = -\alpha_i \qquad \text{for all } i \text{ (indices taken modulo $2t$). }
\]
In particular, the sequence $(\alpha_1,\ldots,\alpha_{2t})$ becomes a repetition of a block of 
length $2m$, so the total length $2t$ must be an integer multiple of $2m$. Therefore, $m \mid t$.
\end{proof}
\begin{cor}\label{2.2}
If, in the lemma above, the conjugacy class $\{p\}$ is primitive (that is,
$p$ is not a proper power), then necessarily $m=t$.  
\end{cor}

Now, our aim is to count words in the group $\mathbb Z_2 * \mathbb Z_3$.
To this end, we encode $(AB)$--words by binary sequences.
An $(AB)$--word of length $4t$, $$ AB^{\alpha_1}AB^{\alpha_2}\cdots AB^{\alpha_{2t}},\qquad \alpha_i\in\{\pm1\},$$is identified with the binary sequence $(\alpha_1,\ldots,\alpha_{2t})$. Let $X_t$ denote the collection of all binary sequences of length $t$.

 Define, 
 $$\mathcal{T}_{4t}=\bigcup_{m \lvert t}
 \{ AB^{\alpha_1}AB^{\alpha_2}....AB^{\alpha_{2t}} \mid \alpha_{m+i} = -\alpha_i, \; \alpha_i\in\{\pm1\} \text{ for all $i$ (indices modulo $2t$) }\}.$$
 Identifying $\mathcal{T}_{4t}$ with 
$$ T_{2t} =\bigcup_{m|t}\{(\alpha_1,\alpha_2,....,\alpha_{2t}) \mid  \alpha_{m+i}=-\alpha_i, \alpha_i\in\{\pm1\}\text{ for all $i$ (indices modulo $2t$) }\} \subset X_{2t}. $$  Write $T_{2t} = \bigcup_{m\lvert t}T_{2t,m}$,   where 
$$ T_{2t,m} = \{(\alpha_1,\alpha_2,....,\alpha_{2t}) \mid  \alpha_{m+i}=-\alpha_i, \alpha_i\in\{\pm1\}\;\text{ for all $i$ (indices modulo $2t$) }\}.$$ 

Conjugate words of an $(AB)$-word, when expressed in the same $(AB)$-form, correspond to cyclic permutations of even order. Keeping this in mind, and using the established bijection, we define a cyclic action on $X_{2t}$ as follows:
\[
\delta^k(\alpha_1, \alpha_2, \ldots, \alpha_{2t}) 
= (\alpha'_1, \alpha'_2, \ldots, \alpha'_{2t}),
\]
where $\alpha'_i = \alpha_{i-k}$ for all $i \in\{ 1, \ldots, 2t\}$. 
Here, the indices are taken modulo $2t$. 
\begin{lemma}
 For fix $k$, and any $(\alpha_1,\alpha_2,....,\alpha_{2t}) \in X_{2t}$ we observe that $$\delta^k(\alpha_1,\alpha_2,....,\alpha_{2t}) \in T_{2t} \hbox{ if and only if }$$ there exists $m$, where $m$ divides $t$, such that $\alpha_{i-k} = -\alpha_{m+i-k}$ for all $i$ (indices modulo $2t$). 

\end{lemma}
\begin{proof}
    Take $$\delta^k(\alpha_1,\alpha_2,....,\alpha_{2t}) = (\alpha'_1,\alpha'_2,....,\alpha'_{2t})$$
    $(\alpha'_1,\alpha'_2,....,\alpha'_{2t}) \in T_{2t}$ if and only if there exists $m$, where $m$ divides $t$, such that $\alpha'_i=-\alpha'_{m+i}$. Now, using the definition of the cyclic action in this, we get $\alpha_{i-k} = -\alpha_{m+i-k}$ where  $m\; \lvert\;t$. 
\end{proof}
 \begin{lemma}
     If $(\alpha_1,\alpha_2,....,\alpha_{2t}) \in T_{2t}$, then $\delta^k(\alpha_1,\alpha_2,....,\alpha_{2t}) \in T_{2t}$ for all $k$.
 \end{lemma}
\begin{proof}
    Since $(\alpha_1,\alpha_2,....,\alpha_{2t}) \in T_{2t}$, it follows that there exists $m$ such that $\alpha_{m+i}=-\alpha_{i}$ where $m$ divides $t$. Now,
    $$\delta^k(\alpha_1,\alpha_2,....,\alpha_{2t}) = (\alpha'_1,\alpha'_2,....,\alpha'_{2t}) \text{, where }\alpha'_{i}=\alpha_{i-k}.$$ That implies, $$ \alpha'_{i} = \alpha_{i-k} = -\alpha_{m+i-k}=-\alpha'_{m+i}.$$
    Hence, $\alpha'_i=-\alpha'_{m+i}, \text{ where } m\;\lvert\;t .$
\end{proof}

We can see that there is a bijection between $\mathcal{T}_{4t}$ and $T_{2t}$ given by $$ \underbrace{\scriptstyle AB^{\alpha_1}AB^{\alpha_2}...AB^{\alpha_m}AB^{-\alpha_1}AB^{-\alpha_2}...AB^{-\alpha_m}.......AB^{\alpha_1}AB^{\alpha_2}...AB^{\alpha_m}AB^{-\alpha_1}AB^{-\alpha_2}...AB^{-\alpha_m}}_{4t \;\text{length}} $$$$\mapsto  \underbrace{\scriptstyle (\alpha_1, \alpha_2,...,\alpha_m,-\alpha_1,-\alpha_2,...,-\alpha_m,.......,\alpha_1,\alpha_2,...,\alpha_m,-\alpha_1,-\alpha_2,...,-\alpha_m)}_{2t \;\text{length}}.$$
We consider the cyclic permutation $\delta$ restricting it to $T_{2t}$, i.e. $\delta: T_{2t} \xrightarrow{} T_{2t}$, given by $(\alpha_1,\alpha_2,....,\alpha_{2t}) \mapsto (\alpha_{2t},\alpha_1,....,\alpha_{2t-1}).$ Thus, the group generated by $\delta$ is cyclic of order $2t$. First, we will compute the cardinality of $T_{2t}$, that is, $\lvert\;T_{2t}\; \rvert$. Once we have that, we can immediately obtain the cardinality of $\mathcal{T}_{4t}$, since there is a bijection between them. \\\\Our goal is to analyze the growth rate of inert classes in $\Gamma$, where inert classes in $ \Gamma$ is identified with the quotient $T_{2t}/\langle\delta\rangle$. Take the prime factorization of $t$. Suppose that the prime factorization of $t$ is  $$t = 2^ap_1^{s_1}p_2^{s_2}.....p_n^{s_n},$$ where $ p_j$'s are distinct odd primes and $ a, s_j \in \mathbb{Z}^+ \cup \{0\}\; \forall\; j$, then every divisor $m$ of $t$ can be written in the form $$ m = 2^lp_1^{k_1}p_2^{k_2}.....p_n^{k_n},$$ where $0\leq l\leq a \text{ and } 0\leq k_j\leq s_j ,\;k_j  \in \mathbb{Z} \;\;\forall\;j \in \{1,...,n\}.$ Suppose that $x=(\alpha_1,\alpha_2,....,\alpha_{2t}) \in T_{2t,m}$, then $\alpha_{m+i}=-\alpha_i$ for all $i$. Hence, $$\alpha_{2m+i}=\alpha_{m+m+i}=-\alpha_{m+i}=\alpha_i, \;\alpha_{3m+i}=\alpha_{m+m+m+i}=-\alpha_{m+m+i}=\alpha_{m+i}=-\alpha_i.$$ Therefore, we see that $$\alpha_{bm+i}=(-1)^b\alpha_i \;\;\;\;\forall\; i, b \in \mathbb{Z}.$$ 
\begin{lemma}\label{2.5}
$\hangindent=2em$
\begin{itemize}
    \item[(1)] For any $l_1, l_2$ with $l_1 \ne l_2$ and $0 \le l_1, l_2 \le a$, we have
    \[
        T_{2t,2^{l_1}} \cap T_{2t,2^{l_2}} = \emptyset.
    \]

    \item[(2)] For any $l$ and $j$ with $0 \le l \le a$ and $k_j \in \{1,\ldots,s_j\}$ for all $j \in \{1,\ldots,n\}$, we have
    \[
        T_{2t,2^l} \cap T_{2t,(p_j^{k_j})} = \emptyset.
    \]

    \item[(3)] For any $j_1, j_2 \in \{1,2,\ldots,n\}$ with $j_1 \ne j_2$, we have
    \[
        T_{2t,(p_{j_1}^{k_{j_1}})} \cap T_{2t,(p_{j_2}^{k_{j_2}})}
        = \{ (\alpha_1,\alpha_2,\ldots,\alpha_{2t}) \mid \alpha_{i+1} = -\alpha_i,\; \alpha_i \in \{-1,1\} \}.
    \]

    \item[(4)] For each $j \in \{1,2,\ldots,n\}$ and all $k_j, k'_j\in \{0,1,\ldots,s_j\}$ where $k_j \leq k'_j$ , we have
    \[
        T_{2t,(p_j^{k_j})} \subseteq T_{2t,(p_j^{k'_j})}.
    \]
\end{itemize}
\end{lemma}
\begin{proof}
\begin {enumerate}
\item Suppose that there exist $x = (\alpha_1,\alpha_2,....,\alpha_{2t}) \in T_{2t, 2^{l_1}} \cap T_{2t, 2^{l_2}}$, then $\alpha_{2^{l_!}+i}=-\alpha_i$ and $\alpha_{2^{l_2}+i}=-\alpha_i$ for all $i$. Take $l_1 >l_2$, so we can write $l_1 = l_2+h $ for some $h \in \mathbb{Z}^+$. Now $\alpha_{bm +i}=(-1)^b \alpha_i$ where $b \in \mathbb{Z}^+$ for all $i$, therefore $\alpha_{2^h 2^{l_2} +i} =(-1)^{2^h}\alpha_i=\alpha_i$, which implies $\alpha_{2^{l_1}+i}=\alpha_i$, a contradiction.

\item Fix $l, j$ and suppose that there exists $x = (\alpha_1,\alpha_2,....,\alpha_{2t}) \in T_{2t, 2^{l}} \cap T_{2t, (p_j^{k_j})}$, then $\alpha_{2^l + i}=-\alpha_i$ and $\alpha_{({p_j}^{k_j})+i}=-\alpha_i$ for all $i$. Since $\alpha_{bm+i}=(-1)^b \alpha_i$  for all $i$ implies $\alpha_{({p_j}^{k_j})\cdot2^l+i}=(-1)^{({p_j}^{k_j})}\alpha_i=-\alpha_i$ and $\alpha_{2^l \cdot ({p_j}^{k_j})+i}=(-1)^{2^l}\alpha_i=\alpha_i$, which gives $\alpha_i=-\alpha_i$ for all $i$, a contradiction.
\item Take $ {p_{j_1}}^{k_{j_1}} = p$ and ${p_{j_2}}^{k_{j_2}}=q$, both are odd numbers and $\gcd(p,q)=1$, so there are some integers $c, d$ such that $cp+dq=1$. Suppose that $x = (\alpha_1,\alpha_2,....,\alpha_{2t}) \in T_{2t,p} \cap T_{2t, q}$, then $\alpha_{p+i}=-\alpha_i$ and $\alpha_{q+i}=-\alpha_i$ for all $i$. Using $\alpha_{bm+i}=(-1)^b\alpha_i$, we get $$ \alpha_{cp+i}=(-1)^c\alpha_i \imp \alpha_{cp+dq+i}=(-1)^c\alpha_{dq+i} =(-1)^{c+d}\alpha_i \imp \alpha_{1+i}=(-1)^{c+d}\alpha_i$$for all  $i$.
    Now, we know that for any odd number $y$, $y \equiv 1 \pmod 2$. Since $p$ and $q$ are both odd numbers, we have$$ p \equiv 1\pmod 2 \text{ \;\;and \;\;} q \equiv 1\pmod 2. $$ 
    $$ \text{ This implies,\; } cp+dq \equiv c+d\pmod2 \imp 1 \equiv c+d \pmod2. $$Therefore, $c+d$ is an odd integer. Hence, we have $\alpha_{i+1}=-\alpha_i$ for all $i$.
\item First we fix $j$, now let $x= (\alpha_1,\alpha_2,....,\alpha_{2t})\in T_{2t, (p_j^{k_j})}$, then $\alpha_{(p_j^{k_j})+i}=-\alpha_i$ for all $i$. Since $k_j \in \mathbb{Z}$ and $0 \leq k_j \leq k'_j\leq s_j$, there are some $h \in \mathbb{Z}$ such that $k_j + h = k'_j$. We know that $\alpha_{bm +i}=(-1)^b\alpha_i$ for all $i$ and $b\in \z $, therefore $$ \alpha_{(p_j^h)(p_j^{k_j})+i}=(-1)^{p_j^h}\alpha_i=-\alpha_i, \text{ that is, \;}  \alpha_{(p_j^{k'_j})+i}=-\alpha_i  \text{\; for all $i$ }.$$ 
\end{enumerate}
This means $x \in T_{2t, (p_j^{k'_j})}$. 
\end{proof}

\begin{lemma}\label{2.6}
    Let $$M=\{ \;\prod_{j=1}^{n}p_j^{k_j} \mid \;k_j \in\{0, 1,...., s_j\} \;\;\forall\; j \;\}.$$ Then, for each $y \in M$ and $l_1,\,l_2 \in \{0, 1,2,..., a\},$ 
\begin{enumerate}
\item $T_{2t,y} \subseteq T_{2t,p_1^{s_1}p_2^{s_2}....p_n^{s_n}}.$
\item $T_{2t,2^{l_1}y} \subseteq T_{2t,2^{l_1}(p_1^{s_1}p_2^{s_2}....p_n^{s_n})}.$
\item for any $l_1 \neq l_2 \in \{1, 2,..., a\}$ and any $y, z \in M,$
     $$T_{2t,(2^{l_1})y} \cap T_{2t,(2^{l_2})z} = \emptyset.$$ 

\end{enumerate}
\end{lemma}
\begin{proof}
\begin{enumerate}
    \item Assume that $x =(\alpha_1,\alpha_2,....,\alpha_{2t})\in T_{2t,y}$ for some fixed $y \in M$, implies $\alpha_{y+i}=-\alpha_i$ for all $i$, take $y=p_1^{k_1}p_2^{k_2}....p_n^{k_n}$ for some fixed $k_j$'s. Since $k_j \in\{0,1,..., s_j\}$ for all j, there exists $h_j \in \mathbb{Z}$ such that $k_j +h_j =s_j $ for all $j$. Now, using $\alpha_{bm +i}=-\alpha_i$ here, we get
        $$ \alpha_{(p_1^{h_1}p_2^{h_2}....p_n^{h_n})\cdot y + i} =(-1)^{(p_1^{h_1}p_2^{h_2}....p_n^{h_n})}\alpha_i=-\alpha_i \imp \alpha_{(p_1^{s_1}p_2^{s_2}....p_n^{s_n}) + i} = -\alpha_i \;\;\text{ for all}\;i.$$ Hence, $x \in T_{2t,p_1^{s_1}p_2^{s_2}....p_n^{s_n}}.$

    \item Fix $l_1$, and let $x =(\alpha_1,\alpha_2,....,\alpha_{2t}) \in T_{2t, 2^{l_1}y}$ be for some fixed $y \in M$. It follows that $\alpha_{2^{l_1}y+i}=-\alpha_i$ for all $i$. Take $y=p_1^{k_1}p_2^{k_2}....p_n^{k_n}$ with fixed integers $k_j's$. As $k_j \in\{0,1,..., s_j\}$ for each $j$, we can find $h_j \in \mathbb{Z}$ satisfying $k_j +h_j =s_j $ for all $j$. Now,
    $$\alpha_{(p_1^{h_1}p_2^{h_2}....p_n^{h_n})\cdot 2^{l_1}y + i} =(-1)^{(p_1^{h_1}p_2^{h_2}....p_n^{h_n})}\alpha_i=-\alpha_i \imp \alpha_{2^{l_1}(p_1^{s_1}p_2^{s_2}....p_n^{s_n})} = -\alpha_i \;\;\forall\;i.$$ 
    Hence, $x \in T_{2t,2^{l_1}p_1^{s_1}p_2^{s_2}....p_n^{s_n}}.$

    \item Take $l_1>l_2$, so there exists some $q \in \mathbb{Z} $ such that $ l_1=l_2+q$. Suppose that $x=(\alpha_1,\alpha_2,....,\alpha_{2t}) \in T_{2t,(2^{l_1})y} \cap T_{2t,(2^{l_2})z}$, then $\alpha_{(2^{l_1})y +i }= -\alpha_i$ and $\alpha_{(2^{l_2})z + i}=-\alpha_i$ for all $i$. Take $y=p_1^{k_1}p_2^{k_2}....p_n^{k_n}$ and $z=p_1^{k'_1}p_2^{k'_2}....p_n^{k'_n}$ where each $k_j, k'_j \in \{0,1,.., s_j\}$ for all $j$, so we can find $h_j,\, h'_j \in \mathbb{Z}$ that satisfies $k_j + h_j=s_j$ and $k'_j+h'_j=s_j$ for each $j$. Therefore,
    \begin{align*}\alpha_{(p_1^{h_1}p_2^{h_2}....p_n^{h_n})\cdot (2^{l_1})y + i} &=(-1)^{(p_1^{h_1}p_2^{h_2}....p_n^{h_n})}\alpha_i\\
    \imp \alpha_{(2^{l_1})(p_1^{s_1}p_2^{s_2}....p_n^{s_n})} &= -\alpha_i \;\;\forall\;i.
    \end{align*}And also,
    \begin{align*}
    \alpha_{(p_1^{h'_1}p_2^{h'_2}....p_n^{h'_n})\cdot (2^l_{2})z + i} & =(-1)^{(p_1^{h'_1}p_2^{h'_2}....p_n^{h'_n})}\alpha_i \\
    \imp\; \alpha_{(2^{l_2})(p_1^{s_1}p_2^{s_2}....p_n^{s_n})} &= -\alpha_i \\
    \imp\; \alpha_{2^q\cdot2^{l_2}(p_1^{s_1}p_2^{s_2}....p_n^{s_n})}&=(-)^{2^q}\alpha_i=\alpha_i\;\;\forall\;i.\\
    \imp\; \alpha_{(2^{l_1})(p_1^{s_1}p_2^{s_2}....p_n^{s_n})} &= \alpha_i\;\;\forall\;i, \text{ which gives a contradiction. }
    \end{align*} \end{enumerate}
    Hence, $T_{2t,(2^{l_1})y} \cap T_{2t,(2^{l_2})z} = \emptyset.$ 
\end{proof}

\begin{theorem}\label{2.7}
Suppose that the prime factorization of $t$ is $ t=2^ap_1^{s_1}p_2^{s_2}....p_n^{s_n} $ where $p_j's$ are distinct odd primes and $a, s_j \in \mathbb{Z}^+ \cup \{0\} \;\;\forall \; j.$ Then $$\lvert\;\mathcal{T}_{4t}\;\rvert=\sum_{l=0}^{a}2^{2^lp_1^{s_1}p_2^{s_2}....p_n^{s_n}}.$$
\end{theorem}
\begin{proof}
  Let $x \in T_{2t}$. Here $$ x = \underbrace{\scriptstyle (\alpha_1,\alpha_2,..,\alpha_m,-\alpha_1, -\alpha_2,.., -\alpha_m,.....,\alpha_1, \alpha_2..,\alpha_m,-\alpha_1, -\alpha_2..., -\alpha_m)}_{2t \text{ length }},$$ where $m$ divides $t$ and $\alpha_i \in \{-1,1\}$ for all $i$. For each $\alpha_i$ there are two possible choices, and hence there are exactly $2^{m}$ words of this form. However, this count involves repetitions. Using \lemref{2.5} and \lemref{2.6}, we see that if $k \mid m$, then every element of $T_{2t,k}$ also lies in $T_{2t,m}$, provided that $m$ and $k$ have the same $2-$ power components. Thus, $T_{2t,k} \subseteq T_{2t,m}.$ Furthermore, maximal divisors with different $2$--power components give rise to disjoint sets $T_{2t,m}$. Consequently, each divisor contributes only through its maximal extension, and summing over maximal divisors ensures that each contribution is counted exactly once. Therefore, we have $$\lvert\;T_{2t}\;\rvert=\sum_{l=0}^{a}2^{2^lp_1^{s_1}p_2^{s_2}....p_n^{s_n}}.$$
  Since there is a bijection between $\mathcal{T}_{4t}$ and $T_{2t}$ given by $$ \underbrace{\scriptstyle AB^{\alpha_1}AB^{\alpha_2}...AB^{\alpha_m}AB^{-\alpha_1}AB^{-\alpha_2}...AB^{-\alpha_m}.......AB^{\alpha_1}AB^{\alpha_2}...AB^{\alpha_m}AB^{-\alpha_1}AB^{-\alpha_2}...AB^{-\alpha_m}}_{4t \;\text{length}} $$$$\mapsto  \underbrace{\scriptstyle (\alpha_1, \alpha_2,...,\alpha_m,-\alpha_1,-\alpha_2,...,-\alpha_m,.......,\alpha_1,\alpha_2,...,\alpha_m,-\alpha_1,-\alpha_2,...,-\alpha_m)}_{2t \;\text{length}}$$ 
  Therefore, $\lvert\;\mathcal{T}_{4t}\;\rvert=\sum_{l=0}^{a}2^{2^lp_1^{s_1}p_2^{s_2}....p_n^{s_n}}$.
\end{proof}

\begin{remark}
In order to count the sets $T_{2t,m}$ without overcounting, it is necessary to
understand how these sets can overlap. Since for some divisors $m \mid t$ the corresponding sequences intersect, we first identify the divisors that
contain all smaller overlapping ones. Writing
$t = 2^{a} p_{1}^{s_{1}} \cdots p_{n}^{s_{n}}$, the relation
$\alpha_{m+i} = -\alpha_{i}$ yields the iterated identity
$\alpha_{bm+i} = (-1)^{b}\alpha_{i}$, which is the key to describing these
intersections. \lemref{2.5} and \lemref{2.6} show that once a maximal divisor $m = 2^{\ell} p_{1}^{s_{1}} \cdots p_{n}^{s_{n}}$ is chosen, every
$T_{2t,k}$ with $k \mid m$ is contained in $T_{2t,m}$, while the maximal divisors
corresponding to different $2$-power components are disjoint. For example,
$T_{2t,p_{1}}$ is contained in $T_{2t,p_{1}p_{2}}$, and continuing in this
manner leads to the unique largest odd divisor $p_{1}^{s_{1}} \cdots p_{n}^{s_{n}}$.
Thus, each divisor $m = 2^{\ell} p_{1}^{j_{1}} \cdots p_{n}^{j_{n}}$ contributes
only through its maximal extension, and the sets $T_{2t,m}$ for different
values of~$\ell$ do not overlap. This explains why the summation in
Theorem~2.12 involves only the maximal divisors
$2^{\ell}p_{1}^{s_{1}} \cdots p_{n}^{s_{n}}$, ensuring that each contribution is
counted exactly once.
\end{remark}

Consider the map
$$
\beta : \mathcal{T}_{4t} \longrightarrow \mathcal{T}_{4t}, \qquad\beta\!\left(AB^{\alpha_1}\, AB^{\alpha_2} \cdots AB^{\alpha_{2t}}\right)= AB^{\alpha_{2t}}\, AB^{\alpha_1} \cdots AB^{\alpha_{2t-1}},
$$ which cyclically shifts the $2t$ consecutive $AB$-blocks. Recall that
$$
\mathcal{T}_{4t}=\bigcup_{m \lvert t}\{ AB^{\alpha_1}AB^{\alpha_2}....AB^{\alpha_{2t}} \mid \alpha_{m+i} = -\alpha_i\; \text{for all i where} \; \alpha_i \in \{-1, 1\}\}.
$$
With this structure, the action of $\beta$ rotates the sign pattern 
$(\alpha_1,\ldots,\alpha_{2t})$ while preserving the defining relation 
$\alpha_{m+i} = -\alpha_i$. Therefore, the cyclic group $\langle \beta \rangle$ acts on $\mathcal{T}_{4t}$, and the quotient $$\mathcal{T}_{4t}/\langle \beta \rangle = \mathscr{I}_{4t}$$ is the set of inert classes. Since $\beta$ induces a cyclic permutation of the $2t$ blocks, the group $\langle \beta \rangle$ is cyclic in order $2t$.
\begin{theorem}\label{2.9}
Let the prime factorization of t be $ t=2^ap_1^{s_1}p_2^{s_2}....p_n^{s_n} $ where $p_j's$ are distinct odd primes and $a, s_j \in \mathbb{Z}^+ \cup \{0\} \;\;\forall \; j.$ Define $$W'_t= \{\;m = 2^{l}p_1^{j_1}p_2^{j_2}\cdots p_n^{j_n}\mid 0\leq j_i \leq s_i, \;0\leq l\leq a\;\forall\;\; i\},$$ and, for each $m \in W$, define $$W_{t}^{'m}= \{ \;k = 2^{l}p_1^{k_1}p_2^{k_2}\cdots p_n^{k_n}\mid 0\leq k_i \leq j_i \;\;\forall \; i\}.$$ 
Then
$$ \lvert \;\mathscr{I}_{4t}\;\rvert = \sum_{m \in W'_t} \left ( \frac{\lvert \,T_{2t,m} \,\rvert- \lvert\, \cup_{k\in W_t^{'m},\, k\neq m}T_{2t,k}\,\rvert}{2m}\right).$$
\end{theorem}
\begin{proof}
 Let $\{p'\}\in T_{2t}/\langle\delta\rangle$. let $p' \in T_{2t}$ be a representative. Suppose that $p'$ is of the form $$ p' = \underbrace{\scriptstyle (\alpha_1,\alpha_2,..,\alpha_m,-\alpha_1, -\alpha_2,.., -\alpha_m,.....,\alpha_1, \alpha_2..,\alpha_m,-\alpha_1, -\alpha_2..., -\alpha_m)}_{2t \text{ length }},$$ where $m$ divides $t$ and $\alpha_i \in \{-1,1\}.$ For each $m$, and for each element in $T_{2t,m}$, there are at most $2m$ conjugate elements in $T_{2t,m}$, and $2^m$ elements of this type. Conjugate elements corresponding to different divisors $m$ may coincide. Consequently, in order to enumerate distinct conjugacy classes in $T_{2t}/\langle \delta \rangle$,  we first identify all such overlaps arising
 from different divisors $m$ of $t$, and then exclude the repeated elements.

Define $W'_t= \{\;m = 2^{l}p_1^{j_1}p_2^{j_2}\cdots p_n^{j_n}\mid 0\leq j_i \leq s_i, \;0\leq l\leq a\;\forall\;\; i\}$, and, for each $m \in W$, define $W_{t}^{'m}= \{ \;k = 2^{l}p_1^{k_1}p_2^{k_2}\cdots p_n^{k_n}\mid 0\leq k_i \leq j_i \;\;\forall \; i\}.$ Using \lemref{2.5} and \lemref{2.6}, we see that for every 
\[
m = 2^{\,l} p_1^{j_1} p_2^{j_2} \cdots p_n^{j_n},
\qquad 0 \le j_i \le s_i,\; 0 \le l \le a,
\]
each $k \in W_{t}^{'m}$ where $k\neq m$  produce repetitions among its conjugate elements, we have
\[
T_{2t,k} \subset 
T_{2t,m},
\]
Therefore, by removing all elements of 
\[
T_{2t,k} \text{ where }  k \neq m
\]
from $T_{2t,m}$, we obtain exactly $2m$ conjugate elements for each element in 
\[
T_{2t,m} \setminus \cup _{k\in W_t^{'m}, k \neq m}T_{2t,k}.
\]
So for each element $p \in T_{2t,m} \setminus
\cup _{k\in W_t^{'m}, k \neq m}T_{2t,k}$ , the conjugacy classes in $T_{2t,m} \setminus
\cup _{k\in W_t^{'m}, k \neq m}T_{2t,k}$ is 

$$ \left ( \frac{\lvert\, T_{2t,m} \,\rvert- \lvert \,\cup_{k\in W_t^{'m},\, k\neq m}T_{2t,k}\,\rvert}{2m}\right).$$ 
The number of conjugacy classes in $T_{2t}$ is $$ \lvert\; T_{2t} \;/ \langle\delta\rangle\;\rvert = \sum_{m \in W'_t} \left ( \frac{\lvert \,T_{2t,m}\, \rvert- \lvert\, \cup_{k\in W_t^{'m},\, k\neq m}T_{2t,k}\,\rvert}{2m}\right).$$ \\Since there is a bijection between $\mathcal{T}_{4t}$ and $T_{2t}$, therefore, there is a bijection between \\$\mathscr{I}_{4t}=\mathcal{T}_{4t}\;/\langle\beta\rangle$ and $T_{2t} \;/ \langle\delta\rangle$. Hence, 
$$  \lvert \,\mathscr{I}_{4t}\,\rvert = \sum_{m \in W'_t} \left (\frac{\lvert\, T_{2t,m} \,\rvert- \lvert\, \cup_{k\in W_t^{'m},\, k\neq m}T_{2t,k}\,\rvert}{2m}\right).$$  
This completes the proof. 
\end{proof}

\medskip 

We denote by $\mathscr{I}^p_{4t}$ the set of all primitive inert classes in $\mathscr{I}_{4t}$ and $\mathscr{I}^{np}_{4t}$ the set of all non-primitive inert classes in $\mathscr{I}_{4t}$.  From \corref{2.2} we know that a conjugacy class is primitive then necessarily $m = t$. Therefore by applying \thmref{2.9}, we can determine the number of primitive inert classes in $\mathscr{I}_{4t}$. 
We now state the corresponding corollary.
\begin{cor} \label{2.10}Specializing \thmref{2.9} in case $m=t$, we obtain the following result,
  \begin{enumerate}
  \item if there exist non trivial odd primes $p_i$ in t, then $\lvert \,\mathscr{I}^p_{4t}\,\rvert =\left ( \frac{\lvert\, T_{2t,t} \,\rvert- \lvert\,\cup_{k\in W_t^{'t}, k \neq t} T_{2t,k}\,\rvert}{2t}\right).$ 
  \item if $t$ is of the form $t =2^l$, then  $\lvert \,\mathscr{I}^p_{4t}\,\rvert = \frac{2^{t}}{2t}=\left (\frac{\lvert\, T_{2t,t} \,\rvert}{2t}\right)$.
\end{enumerate}
\end{cor}
\begin{example} 
We illustrate the above discussion with the case $t=3$. The prime factorization of $t$ is 
$$
t=3=2^{0}\cdot 3^{1},
$$
so $a=0$ and the unique maximal odd divisor is $3$. The divisors of $t$ are $m=1,3$.  
By \lemref{2.5} and \lemref{2.6}, we have 
$$
T_{6,1}\subset T_{6,3},
$$
and hence only the maximal divisor $m=3$ contributes new elements.
The defining relation for $T_{6,3}$ is
$$
\alpha_{3+i}=-\alpha_i \quad (i=1,2,3),
$$
which implies that every element of $T_{6,3}$ is uniquely determined by its first three entries and has the form
$$
(\alpha_1,\alpha_2,\alpha_3,-\alpha_1,-\alpha_2,-\alpha_3), \qquad \alpha_i \in \{\pm1\}.
$$
Thus there are $2^3 = 8$ elements in total, namely
$$
\begin{aligned}
&(1,1,1,-1,-1,-1), \quad (1,1,-1,-1,-1,1), \quad (1,-1,1,-1,1,-1), \quad (1,-1,-1,-1,1,1),\\
&(-1,1,1,1,-1,-1), \quad (-1,1,-1,1,-1,1), \quad (-1,-1,1,1,1,-1), \quad (-1,-1,-1,1,1,1).
\end{aligned}
$$

For $m=1$, the relation $\alpha_{i+1}=-\alpha_i$ implies strict alternation, yielding
$$
T_{6,1}=\{(1,-1,1,-1,1,-1), \, (-1,1,-1,1,-1,1)\}.
$$
Both of these sequences appear in the above list, confirming explicitly that $T_{6,1}\subset T_{6,3}$.  
Using \thmref{2.7}, we have $\lvert\,T_6\,\rvert=8$ which agrees with the explicit counting.  
\\
In \thmref{2.9} the first sum for $m=3$ gives $(\lvert\,T_{6,3}\,\rvert-\lvert\,T_{6,1}\,\rvert)/(2\cdot 3) = (8-2)/6 = 1$, and the second sum for $m=1$ gives $\lvert\,T_{6,1}\,\rvert/(2\cdot 1) = 2/2 = 1$.  
Hence, 
$$
\lvert\,\mathscr{I}_{12}\,\rvert = 1 + 1 = 2,
$$
showing that there are exactly two conjugacy classes under cyclic permutation.

Explicitly, the set $T_6 = T_{6,3} = \{ (\alpha_1,\alpha_2,\alpha_3,-\alpha_1,-\alpha_2,-\alpha_3) \}$ decomposes as follows: the alternating sequences $T_{6,1} = \{(1,-1,1,-1,1,-1), \, (-1,1,-1,1,-1,1)\}$ form a single conjugacy class under cyclic shifts, while the remaining six sequences
$$ \begin{aligned}
&(1,1,1,-1,-1,-1), \quad (1,1,-1,-1,-1,1), \quad (1,-1,1,-1,1,-1), \\
& (1,-1,-1,-1,1,1), \quad (-1,1,1,1,-1,-1), \quad (-1,-1,1,1,1,-1)
\end{aligned}
$$
form the other conjugacy class. Thus, under cyclic permutation, $T_6$ splits into exactly two conjugacy classes: one represented by the alternating word $(1,-1,1,-1,1,-1)$ and the other by $(1,1,1,-1,-1,-1)$. This confirms the count obtained from the formula and illustrates how the subtraction term prevents overcounting.
\end{example}
 We now investigate the asymptotic behavior of the number of inert classes in $\Gamma$ as $t$ tends to infinity.
\begin{theorem}\label{2.11}
   As $t \to \infty$,  $\lvert\;\mathcal{T}_{4t}\;\rvert \sim 2^t$. 
\end{theorem}
\begin{proof}
    From Theorem~\ref{2.7} we have 
$$\lvert\;\mathcal{T}_{4t}\;\rvert=\sum_{l=0}^{a}2^{2^lp_1^{s_1}p_2^{s_2}....p_n^{s_n}}$$ for the prime factorization of $t=2^ap_1^{s_1}p_2^{s_2}....p_n^{s_n}$, where $a_j,\, s_j \in \mathbb{Z}^+\cup\{0\}$. Thus, we may write $\frac{t}{2^a}=p_1^{s_1}p_2^{s_2}....p_n^{s_n} $ in the above expression,
$$\lvert\;\mathcal{T}_{4t}\;\rvert=\sum_{l=0}^{a}2^{2^l\cdot\frac{t}{2^a}}=\sum_{l=0}^{a}2^{\frac{t}{2^{a-l}}}$$
$$\lvert\;\mathcal{T}_{4t}\;\rvert=\sum_{j=0}^{a}2^{\frac{t}{2^{j}}}$$ Let $\nu(t)$ denote the highest power of $2$ dividing $t$ (that is, the 2-adic valuation of $t$), that is, $\nu(t)=a.$ Then
$$ \lvert\;\mathcal{T}_{4t}\;\rvert=\sum_{j=0}^{\nu(t)}2^{(\frac{t}{2^{j}})}=2^t+\sum_{j=1}^{\nu(t)} 2^t\cdot2^{(\frac{t}{2^{j}}-t)}=2^t(1+R)$$ where $$R=\sum_{j=1}^{\nu(t)}2^{\frac{t}{2^{j}}-t}=\sum_{j=1}^{\nu(t)}2^{-{(t-\frac{t}{2^{j}})}}=\sum_{j=1}^{\nu(t)}\frac{1}{2^{(t-\frac{t}{2^{j}})}}$$
$$R \,\leq\, \nu(t)\frac{1}{2^{(t-\frac{t}{2})}}=\nu(t)2^{-\frac{t}{2}}$$
Since $\nu(t) < \log_{2}(t) \ll 2^{t/2}$, we obtain
\[
0 \le R \le \frac{\log_{2}(t)}{2^{t/2}}.
\]
as $t \to \infty$, the right-hand side tends to $0$, and hence $R \to 0$. Therefore,
\[
\frac{\lvert\,\mathcal{T}_{4t}\,\rvert} {2^{\,t}}\to 1 \quad \text{as } t \to \infty,
\]and consequently, $\lvert\,\mathcal{T}_{4t}\,\rvert \sim 2^{\,t}.$
\end{proof}
\begin{theorem}\label{2.12} As $t \to \infty$, 
   \begin{enumerate}
   \item $\lvert\,\mathscr{I}_{4t}\,\rvert\, \sim \frac{2^{t-1}}{t}$, 
   \item $\lvert\,\mathscr{I}_{\leq4t}\,\rvert\, \sim \frac{2^{t}}{t+1}$. 
   \end{enumerate}
\end{theorem}
\begin{proof}
\begin{enumerate}
\item  From \thmref{2.9}, we have
   $$\lvert \;\mathscr{I}_{4t}\;\rvert = \sum_{m \in W'_t} \left ( \frac{\lvert\, T_{2t,m} \,\rvert- \lvert \,\cup_{k\in W_t^{'m},\, k\neq m}T_{2t,k}\,\rvert}{2m}\right)\leq\sum_{m \in W_t'} \left ( \frac{\lvert \,T_{2t,m}\, \rvert}{2m}\right)=\sum_{m|t}\frac{2^m}{2m}.$$ 
   Now, observe
   $$\bigg(\frac{2^t-\sum_{k \lvert t, k \neq t}2^k}{2t}\bigg)\leq\lvert \;\mathscr{I}_{4t}\;\rvert \leq \frac{2^t}{2t}+\sum_{m|t,\, m \neq t}\frac{2^m}{2m}.$$
   \begin{equation}\label{*} 
       \frac{2^t}{2t}-E' \leq \lvert \;\mathscr{I}_{4t}\;\rvert \leq \frac{2^t}{2t}+E
        \end{equation}
   where  $ E= \sum_{m|t,\, m \neq t}\frac{2^m}{2m} $ and $E'=\sum_{k|t,\, k \neq t}\frac{2^k}{2t}$. \\Now, since the largest proper divisor of $t$ is $t/2$,
   $$ E= \sum_{m|t,\, m \neq t}\frac{2^m}{2m} \leq \sum_{m|t,\, m \neq t}\frac{2^m}{2} \leq \sum_{m|t,\, m \neq t}\frac{2^{t/2}}{2}< \varphi(t)\frac{2^{t/2}}{2}$$  
   $$E\,<\,\frac{t\,2^{t/2}}{2} \quad ( \varphi(t)\,\leq\, t)$$
   Dividing the above inequality by $\frac{2^t}{2t},$ we get
   $$\frac{E}{\frac{2^t}{2t}}\,\leq \,\frac{\frac{t\,2^{t/2}}{2}}{\frac{2^t}{2t}} = \frac{t^2}{2^{t/2}}, \text{ as } t \to \infty \quad  \frac{E}{\frac{2^t}{2t}}\to 0$$ because an exponential function $2^{t/2}$ grows much faster than any polynomial such as $t^2.$
   and 
   $$ E'=\sum_{\substack{k \lvert t \\ k \neq t}} \frac{2^k}{2t}  \leq \sum_{\substack{k \lvert t \\ k \neq t}}\frac{2^k}{2} <\varphi(t)\frac{2^{t/2}}{2}$$
   $$E'<t\frac{2^{t/2}}{2}$$
   Dividing the above inequality by $\frac{2^t}{2t},$ we get
   $$ \frac{E'}{\frac{2^t}{2t}}\,< \,\frac{t\frac{2^{t/2}}{2}}{\frac{2^t}{2t}}=\frac{t^2}{2^{t/2}}\text{ as } t \to \infty \quad  \frac{E'}{\frac{2^t}{2t}}\to 0$$
   From \eqnref{*}, 
$$1-\frac{E'}{\frac{2^t}{2t}}\leq\frac{\lvert \;\mathscr{I}_{4t}\;\rvert}{\frac{2^t}{2t}} \leq 1+\frac{E}{\frac{2^t}{2t}}$$
as $t \to \infty$, $\dfrac{E}{\frac{2^t}{2t}}$ and $\dfrac{E'}{\frac{2^t}{2t}}$ both tend to $0$. Therefore,
$$\frac{\lvert \;\mathscr{I}_{4t}\;\rvert}{\frac{2^t}{2t}} \to 1.$$
Hence, $\lvert \;\mathscr{I}_{4t}\;\rvert\, \sim\, \frac{2^{t-1}}{t}.$ 
\item  Note that $$\lvert \;\mathscr{I}_{\leq4t}\;\rvert = \sum_{n=1}^{t}\,\,\lvert\,\mathscr{I}_{4n}\,\rvert=\sum_{n=1}^{t}\sum_{m \in W'_n} \left ( \frac{\lvert\, T_{2n,m} \,\rvert- \lvert\, \cup_{k\in W_n^{'m},\, k\neq m}T_{2n,k}\,\rvert}{2m}\right)\leq\sum_{n=1}^{t}\sum_{m \lvert n}\frac{2^m}{2m}.$$ It follows that
$$\sum_{n=1}^{t}\bigg({\frac{2^n-\sum_{k \lvert n, k \neq n}2^k}{2n}\bigg)}\,\leq\,\lvert \;\mathscr{I}_{\leq4t}\;\rvert \,\leq\,\sum_{n=1}^{t}\frac{2^n}{2n}+\sum_{n=1}^{t}\sum_{\substack{m \lvert n\\m \neq n}} \frac{2^m}{2m}.$$
From (1), we have $$\sum_{m|t,\, m \neq t}\frac{2^m}{2m}\,<\,\frac{t\,2^{t/2}}{2} \quad \text{ and } \quad \sum_{\substack{k \lvert t \\ k \neq t}} \frac{2^k}{2t} <t\frac{2^{t/2}}{2}$$Therefore, $$\sum_{n=1}^{t}\sum_{m|n,\, m \neq n}\frac{2^m}{2m}\,<\,\sum_{n=1}^{t}\frac{n\,2^{n/2}}{2}\,<\,\frac{t^2\,2^{t/2}}{2}$$ and similarly, $$\sum_{n=1}^{t}\sum_{\substack{k \lvert n \\ k \neq n}} \frac{2^k}{2n} <\sum_{n=1}^{t}n\frac{2^{n/2}}{2}\,<\,\frac{t^2\,2^{t/2}}{2}$$ Substituting these inequalities, we obtain
$$\sum_{n=1}^{t}\frac{2^n}{2n}-t^2\frac{2^{t/2}}{4}\,\leq\,\lvert \;\mathscr{I}_{\leq4t}\;\rvert \,\leq\,\sum_{n=1}^{t}\frac{2^n}{2n}+t^2\frac{2^{t/2}}{4}$$
$$1-\frac{t^2\frac{2^{t/2}}{4}}{\sum_{n=1}^{t}\frac{2^n}{2n}}\,\leq\,\frac{\lvert \;\mathscr{I}_{\leq4t}\;\rvert}{\sum_{n=1}^{t}\frac{2^n}{2n}} \,\leq\,1+\frac{t^2\frac{2^{t/2}}{4}}{\sum_{n=1}^{t}\frac{2^n}{2n}}$$
Since $\sum_{n=1}^{t}\frac{2^n}{2n}\,\geq\,\frac{2^t}{t},$ we get
$$1-\frac{t^2\frac{2^{t/2}}{4}}{\frac{2^t}{2t}}\,\leq\,\frac{\lvert \;\mathscr{I}_{\leq4t}\;\rvert}{\sum_{n=1}^{t}\frac{2^n}{2n}} \,\leq\,1+\frac{t^2\frac{2^{t/2}}{4}}{\frac{2^t}{2t}},$$
$$1-\frac{t^3}{2\cdot2^{t/2}}\,\leq\,\frac{\lvert \;\mathscr{I}_{\leq4t}\;\rvert}{\sum_{n=1}^{t}\frac{2^n}{2n}} \,\leq\,1+\frac{t^3}{2\cdot2^{t/2}},$$
as $t \to \infty ,$ $\frac{t^3}{2\cdot2^{t/2}} \to 0$, therefore,
$$\frac{\lvert \;\mathscr{I}_{\leq4t}\;\rvert}{\sum_{n=1}^{t}\frac{2^n}{2n}} \to 1 \quad \text{ as  } \,\,t \to \infty.$$
Hence, $\lvert \;\mathscr{I}_{\leq4t}\;\rvert\sim \sum_{n=1}^{t}\frac{2^n}{2n}.$ Using the Stolz-Cesaro theorem \cite{MR2456018}, we can see that $\sum_{n=1}^{t}\frac{2^n}{2n} \sim \frac{2^{t}}{t+1}.$
\end{enumerate}
This gives the desired result. 
\end{proof}
\begin{lemma}\label{2.13}
$\hangindent=2em$
\begin{enumerate}
    \item $\lvert \;\mathscr{I}^{np}_{4t}\;\rvert\, \leq\, \frac{t}{4}2^{t/2}$,
    \item $\lvert \;\mathscr{I}^{np}_{\leq4t}\;\rvert\, \leq\, \frac{t^2}{4}2^{t/2}$. 
\end{enumerate}
\end{lemma}
\begin{proof}
\begin{enumerate}
\item Using Lemma 2.4 of \cite{combigrowth} we have $$ \lvert \;\mathscr{I}^{np}_{4t}\;\rvert\,=\sum_{s |t }\lvert \;\mathscr{I}^p_{4s}\;\rvert$$ where $s$ are the proper divisors of $t$. Now, using $\corref{2.10}$ in the above expression, 
$$\lvert \;\mathscr{I}^{np}_{4t}\;\rvert\,=\,\sum_{s |t }\left ( \frac{\lvert \,T_{2s,s}\, \rvert- \lvert\,\cup_{k\in W_s^{'s}, k \neq s} T_{2s,k}\,\rvert}{2s}\right)\implies\lvert \;\mathscr{I}^{np}_{4t}\;\rvert\,\leq \,\sum_{s |t } \frac{2^s}{2}.$$
Since the largest proper divisor of $t$ is $t/2$ and there are at most $t/2$ proper divisors, therefore
$$\lvert \;\mathscr{I}^{np}_{4t}\;\rvert\,\leq \,\frac{t}{2} \frac{2^{t/2}}{2} \implies\lvert \;\mathscr{I}^{np}_{4t}\;\rvert\,\leq \,\frac{t}{4}2^{t/2}.$$
Hence, $\lvert \;\mathscr{I}^{np}_{4t}\;\rvert\, \leq\, \frac{t}{4}2^{t/2}.$
  \item from above we have, 
   $$ \lvert \;\mathscr{I}^{np}_{\leq4t}\;\rvert\,=\sum_{n=1}^{t}\,\,\lvert\, \mathscr{I}_{4n}^{np}\,\rvert\,\leq\,\sum_{n=1}^{t}\,\frac{n}{4}2^{n/2}\,\leq\,\sum_{n=1}^{t}\,\frac{t}{4}2^{t/2}\,\leq\,\frac{t^2}{4}2^{t/2}$$
    \end{enumerate}
    Thus, $\lvert \;\mathscr{I}^{np}_{\leq4t}\;\rvert\,\leq\,\frac{t^2}{4}2^{t/2}.$
    \end{proof}
   
\begin{theorem}\label{2.14} As $t \to \infty$, 
   \begin{enumerate}
       \item  $ \lvert\,\mathscr{I}^p_{4t}\,\rvert\, \sim \,\lvert\,\mathscr{I}_{4t}\,\rvert\,\sim\,\frac{2^{t-1}}{t}$, 
       \item $ \lvert\,\mathscr{I}^p_{\leq4t}\,\rvert\, \sim \,\lvert\,\mathscr{I}_{\leq4t}\,\rvert\,\sim\,\frac{2^{t}}{t+1}$. 
       \end{enumerate}
\end{theorem}
\begin{proof}
\begin{enumerate}
    \item    Using \lemref{2.13}, we have
    $$\lvert\,\mathscr{I}_{4t}\,\rvert - \frac{t}{4}2^{t/2}\leq \lvert\,\mathscr{I}^p_{4t}\,\rvert \leq \lvert\,\mathscr{I}_{4t}\,\rvert \quad\quad  (\,\lvert\,\mathscr{I}^p_{4t}\,\rvert=\lvert\,\mathscr{I}_{4t}\,\rvert-\lvert\,\mathscr{I}^{np}_{4t}\,\rvert \,)$$ 
    Dividing by $\lvert\,\mathscr{I}_{4t}\,\rvert$ and observing that $\lvert\, \mathscr{I}_{4t}\,\rvert\geq \bigg(\frac{2^t-\sum_{k \lvert t, k \neq t}2^k}{2t}\bigg)\geq\frac{2^t-\frac{t}{2}2^\frac{t}{2}}{2t} $, it follows that
    $$1 - \frac{\frac{t}{4}2^{t/2}}{\frac{2^t-\frac{t}{2}2^\frac{t}{2}}{2t}}\leq \frac{\lvert\,\mathscr{I}^p_{4t}\,\rvert}{\lvert\,\mathscr{I}_{4t}\,\rvert} \leq 1 $$
    $$1 - \frac{t^2}{2\cdot(2^\frac{t}{2}-\frac{t}{2})}\leq \frac{\lvert\,\mathscr{I}^p_{4t}\,\rvert}{\lvert\,\mathscr{I}_{4t}\,\rvert} \leq 1$$
    Since $\frac{t^2}{2\cdot(2^\frac{t}{2}-\frac{t}{2})} \to 0$ as $t \to \infty$
    \\
    \\
    Hence,  $\lvert\,\mathscr{I}^p_{4t}\,\rvert\, \sim \lvert\,\mathscr{I}_{4t}\,\rvert\,\sim\,\frac{2^{t-1}}{t}.$
    \item for this part, we use the same arguments as above. We have from \lemref{2.13},
    $$\lvert\,\mathscr{I}_{\leq4t}\,\rvert - \frac{t^2}{4}2^{t/2}\leq \lvert\,\mathscr{I}^p_{\leq4t}\,\rvert \leq \lvert\,\mathscr{I}_{\leq4t}\,\rvert$$
     Dividing by $\lvert\,\mathscr{I}_{\leq4t}\,\rvert$ and observing that $\lvert \,\mathscr{I}_{\leq 4t}\,\rvert \geq \sum_{n=1}^{t}\frac{2^n-\frac{n}{2}2^\frac{n}{2}}{2n} \geq \frac{2^t-\frac{t}{2}2^\frac{t}{2}}{2t} $, it follows that
     $$1 - \frac{\frac{t^2}{4}2^{t/2}}{\frac{2^t-\frac{t}{2}2^\frac{t}{2}}{2t} }\leq \frac{\lvert\,\mathscr{I}^p_{\leq4t}\,\rvert}{\lvert\,\mathscr{I}_{\leq4t}\,\rvert} \leq 1 $$
     $$1 - \frac{t^3}{2\cdot(2^{t/2}-\frac{t}{2})}\leq \frac{\lvert\,\mathscr{I}^p_{4t}\,\rvert}{\lvert\,\mathscr{I}_{4t}\,\rvert} \leq 1$$
     Since $\frac{t^3}{2\cdot(2^{t/2}-\frac{t}{2})} \to 0$ as $t \to \infty$
         \end{enumerate}

     Hence,  $\lvert\,\mathscr{I}^p_{\leq4t}\,\rvert\, \sim \lvert\,\mathscr{I}_{\leq4t}\,\rvert\,\sim\,\frac{2^{t}}{t+1}.$
    \end{proof}

\medskip    We are now in a position to prove \thmref{inert}. 
\subsection{Proof of \thmref{inert}}
\begin{enumerate}
    \item From \thmref{2.12}, we have
\[
\lvert\,\mathscr{I}_{\leq4t}\,\rvert\, \sim \frac{2^{t}}{t+1}
\]
as $t \to \infty$. Since $\mathscr{I}_{\leq 2t}$ corresponds to replacing $t$ by $\lfloor \frac{t}{2} \rfloor$ in the above expression, it follows that 
\[
|\mathscr{I}_{\leq2t}| \sim \frac{2^{\lfloor \frac{t}{2} \rfloor}}{\lfloor \frac{t}{2} \rfloor +1}.
\] 
\item From \thmref{2.14}, we know that 
\[ 
\lvert\,\mathscr{I}^p_{\leq4t}\,\rvert\, \sim \,\lvert\,\mathscr{I}_{\leq4t}\,\rvert
\]
Replacing $4t$ by $2t$ in this asymptotic formula, we therefore obtain
\[
\lvert\,\mathscr{I}^p_{\leq2t}\,\rvert\, \sim \,\lvert\,\mathscr{I}_{\leq2t}\,\rvert
\]
which shows that the number of primitive inert classes of length at most $2t$ is asymptotic to the number of inert conjugacy classes of length at most $2t$. \qed
\end{enumerate}

\section{Combinatorial Structure of Ambiguous Classes}
    In this section, we study the asymptotic growth of ambiguous classes in $\Gamma.$
    \begin {lemma}\label{3.1}
    Let $\{p\}$ be an ambiguous class in $\mathscr{A}_{2t}$, and suppose $\{p\}$ contains a cyclically reduced word of the form
    \[
    p \;=\; AB^{\alpha_1}AB^{\alpha_2}\cdots AB^{\alpha_{t}}, 
    \qquad \alpha_i \in \{\pm 1\}.
    \]
    If $\{p\}$ is fixed by $\phi_{Amb}$, then 
     $$ \alpha_i=\alpha_{t+1-i} \text{ for all i }\in\{1,\dots,t\}.$$
     That is, if t is even,
     \[
    p \;=\;
    AB^{\alpha_1}AB^{\alpha_2}\cdots AB^{\alpha_{\frac{t}{2}}}AB^{\alpha_{\frac{t}{2}}}\cdots AB^{\alpha_2}AB^{\alpha_1} ,
    \]
    if t is odd, 
     \[
    p \;=\;
    AB^{\alpha_1}AB^{\alpha_2}\cdots AB^{\alpha_{\lceil \frac{t}{2}\rceil}}\cdots AB^{\alpha_2}AB^{\alpha_1} ,
    \] where $\lceil \frac{t}{2} \rceil$ denotes the least integer that is greater than or equal to the given number $\frac{t}{2}$. 
    \end{lemma}
    \begin{proof}
    Suppose $\{p\}\in \mathscr{A}_{2t}$ and let
$$
p \;\longleftrightarrow\;
\begin{pmatrix}
a & b\\
c & d
\end{pmatrix},
\qquad ad-bc=1,\quad |a+d|>2.
$$
Then
$$
\phi_{Amb}(p)=w^{-1}p^{-1}w
\;\longleftrightarrow\;
\begin{pmatrix}
d & b \\
c & a
\end{pmatrix}.
$$

Moreover,
$$
\phi_{Amb}(A)
\;\longleftrightarrow\;
\phi_{Amb}\!\left(
\begin{pmatrix}
0 & -1\\
1 & 0
\end{pmatrix}\right)
=
\begin{pmatrix}
0 & -1\\
1 & 0
\end{pmatrix}
\;\longleftrightarrow\; A,
$$
and
$$
\phi_{Amb}(B)
\;\longleftrightarrow\;
\phi_{Amb}\!\left(
\begin{pmatrix}
0 & -1\\
1 & 1
\end{pmatrix}\right)
=
\begin{pmatrix}
1 & -1\\
1 & 0
\end{pmatrix}
\;\longleftrightarrow\; ABA.
$$
Consequently,
$$
\phi_{Amb}(A)=A
\qquad\text{and}\qquad
\phi_{Amb}(B)=ABA.
$$
Every element $\gamma \in W$, which is not elliptic, is conjugate to a $(AB)$-word. Hence, each conjugacy class in $\mathscr{A}_{2t}$ contains a representative of the form
\[
p = AB^{\alpha_1} AB^{\alpha_2} \cdots AB^{\alpha_{t}}, \qquad \alpha_i \in \{\pm 1\}.
\]
Since $\phi_{Amb}$ is an anti-homomorphism, thus
$$
\phi_{Amb}(p) = \phi_{Amb}(B^{\alpha_{t}})\phi_{Amb}( A)\cdots  \phi_{Amb}(B^{\alpha_2})\phi_{Amb}(A)\phi_{Amb}(B^{\alpha_1})\phi_{Amb}(A)$$
$$\imp \phi_{Amb}(p)= AB^{\alpha_{t}} \cdots AB^{\alpha_2} AB^{\alpha_1}.
$$
Since $\{p\}\in \mathscr{A}_{2t}$, the conjugacy class $\{p\}$ is invariant by $\phi_{Amb}$. Both words are cyclically reduced. In a free product of two cyclic groups, two cyclically reduced words are conjugate if and only if one is a cyclic permutation of the other. Hence, if $p$ is conjugate to $\phi_{Amb}(p)$, then $\phi_{Amb}(p)$ must be obtained from $p$ by cyclic permutation. Any cyclic permuation that identifies $p$ with $\phi_{Amb}(p)$ must place the initial $A$ of $\phi_{Amb}(p)$ at the beginning of one of these blocks in $p$. Since all exponents in $\phi_{Amb}(p)$ are exactly the reverse of exponents in $p$, the only possible cyclic shift is by exactly $\lceil\frac{t}{2}\rceil$ blocks where $\lceil \frac{t}{2} \rceil$ denotes the least integer that is greater than or equal to the given number $\frac{t}{2}$. This implies that the exponent sequence is symmetric,
\[
 \alpha_i=\alpha_{t+1-i}  \quad \text{for all } i\in \{1,\dots,t\}.
\]
Hence, if $p$ is conjugate to $\phi_{Amb}(p)$, the exponents must satisfy this symmetry condition.
 Therefore, we obtain that $$ \alpha_i=\alpha_{t+1-i} \;\;\text{ for all } ~i\in\{1,2,\dots,t\}.$$\\ This implies if $t$ is even,
     \[
    p \;=\;
    AB^{\alpha_1}AB^{\alpha_2}\cdots AB^{\alpha_{\frac{t}{2}}}AB^{\alpha_{\frac{t}{2}}}\cdots AB^{\alpha_2}AB^{\alpha_1} ,
    \]
    if $t$ is odd, 
     \[
    p \;=\;
    AB^{\alpha_1}AB^{\alpha_2}\cdots AB^{\alpha_{\lceil \frac{t}{2}\rceil}}\cdots AB^{\alpha_2}AB^{\alpha_1}.
    \] 
    This completes the proof. 
\end{proof}
 \vspace{2em}
 \n 
 Now we define the set $\mathcal{U}_{2t}=\{ AB^{\alpha_1}AB^{\alpha_2}....AB^{\alpha_{t}} \mid \alpha_i= \alpha_{t+1-i} \; \text{for all $i$, and} \; \alpha_i=\pm1\}.$
\begin{lemma}\label{3.2}
    For any ambiguous element in \;$\mathcal{U}_{2t}$,
    \begin{enumerate}
        \item If $t$ is even, each ambiguous class consists of exactly two elements, represented by the words $$p=AB^{\alpha_1}AB^{\alpha_2}\cdots AB^{\alpha_{\frac{t}{2}}}AB^{\alpha_\frac{t}{2}}\cdots AB^{\alpha_2}AB^{\alpha_1}$$ and $$p'=AB^{\alpha_{\frac{t}{2}}}\cdots AB^{\alpha_{2}}AB^{\alpha_1}AB^{\alpha_1}AB^{\alpha_2}\cdots AB^{\alpha_{\frac{t}{2}}}.$$
        \item If t is odd, each ambiguous class consists of exactly one element, represented by the word $$p=AB^{\alpha_1}AB^{\alpha_2}\cdots AB^{\alpha_{\lceil \frac{t}{2}\rceil}}\cdots AB^{\alpha_2}AB^{\alpha_1}.$$
        \end{enumerate}
\end{lemma}
\begin{proof}
 Consider \(p \in \mathcal{U}_{2t}\) as an ambiguous element (i.e., \(p\) is conjugate to its reverse word). For cyclically reduced \((AB)\)-words, conjugation is equivalent to taking a cyclic shift of the word.\n Let $t$ be even. Consider the word
  $$
  AB^{\alpha_1} AB^{\alpha_2} \cdots AB^{\alpha_{\frac{t}{2}}} AB^{\alpha_{\frac{t}{2}}} \cdots AB^{\alpha_2} AB^{\alpha_1}.
  $$
  We can represent it by the finite sequence 
  \((\alpha_1, \alpha_2, \ldots, \alpha_{\frac{t}{2}}, \alpha_{\frac{t}{2}}, \ldots, \alpha_2, \alpha_1)\), which is palindromic because it reads the same forward and backward. The conjugacy class of this word contains all its cyclic shifts. However, to remain in the ambiguous permutations of $\alpha_i$'s, we consider only those cyclic permutations that preserve the palindromic structure. Thus, only two the cyclic permutations that give palindromic words are included as elements in the ambiguous class.
   \\Now consider all the cyclic permutations of the finite sequence. Among the $t$ possible cyclic shifts, only the shift by the $0$ positions and the shift by the $\frac{t}{2}$ positions produce palindromic words. All other cyclic shifts break the symmetry and give non-palindromic words. Therefore, within the ambiguous class, there are exactly two palindromic representatives, namely the original word $$p=AB^{\alpha_1}AB^{\alpha_2}\cdots AB^{\alpha_{\frac{t}{2}}}AB^{\alpha_\frac{t}{2}}\cdots AB^{\alpha_2}AB^{\alpha_1}$$ and $$p'=AB^{\alpha_{\frac{t}{2}}}\cdots AB^{\alpha_{2}}AB^{\alpha_1}AB^{\alpha_1}AB^{\alpha_2}\cdots AB^{\alpha_{\frac{t}{2}}}.$$ No other cyclic shift gives a new palindromic representative.
   
   When $t$ is odd, the same argument as above shows that each ambiguous class consists of exactly one palindromic element, represented by the word
   $$p=AB^{\alpha_1}AB^{\alpha_2}\cdots AB^{\alpha_{\lceil \frac{t}{2}\rceil}}\cdots AB^{\alpha_2}AB^{\alpha_1}.$$
This establishes the lemma. 
\end{proof}

\medskip
\begin{lemma}\label{3.3} 

\begin{enumerate}
    \item $ \lvert\, \mathscr{A}_{ 2t}\, \rvert =
\begin{cases}
2^{\frac{t}{2}-1}, & \text{if } t \text{ is even},\\[6pt]
2^{\lceil \frac{t}{2} \rceil}, & \text{if } t \text{ is odd}.
\end{cases}
$
$\hangindent=2em$
    \item $ \lvert\,\mathscr A_{\leq 2t}\,\rvert\simeq2^{t/2}$.
\end{enumerate}
\end{lemma}
\begin{proof}
   \begin {enumerate}
   \item First we consider the case of even $t$. Then there are $2^\frac{t}{2}$ words of the form 
$$
p = AB^{\alpha_1} AB^{\alpha_2} \cdots AB^{\alpha_\frac{t}{2}} AB^{\alpha_\frac{t}{2}} \cdots AB^{\alpha_2} AB^{\alpha_1}
$$
in $\mathcal{U}_{2t}$, we have $\lvert \,\mathcal{U}_{2t} \,\rvert = 2^\frac{t}{2}$.  
From the above \lemref{3.2}, we know that each such word has exactly two conjugates. Therefore, the number of distinct conjugacy classes is
$$
\left\lvert\, \mathcal{U}_{2t}/\langle \delta \rangle\, \right\rvert = \frac{2^\frac{t}{2}}{2}. 
$$ Since \,$\mathcal{U}_{2t}/\langle \delta \rangle = \mathscr{A}_{2t}$, it implies $\lvert\, \mathscr{A}_{4t}\, \rvert = {2^{t-1}}.$
\\If t is odd, there are $2^{\lceil\frac{t}{2}\rceil}$ words of the form 
$$
p = AB^{\alpha_1} AB^{\alpha_2} \cdots AB^{\alpha_{\lceil\frac{t}{2}\rceil}} \cdots AB^{\alpha_2} AB^{\alpha_1}
$$
in $\mathcal{U}_{2t}$. From the above \lemref{3.2}, we know that each such word has exactly one element in it's conjugacy class. Therefore, the number of distinct conjugacy classes is
$$
\left\lvert\, \mathcal{U}_{2t}/\langle \delta \rangle\, \right\rvert = \lvert\, \mathscr{A}_{2t}\,\rvert= 2^{\lceil\frac{t}{2}\rceil}.$$ 
  \item The second one follows: 
    We have
$$
\lvert \,\mathscr A_{\leq 2t}\,\rvert=\sum_{n=1}^{t}\lvert\, \mathscr A_{2n}\,\rvert.
$$
For every \(n\),
$$
\lvert\, \mathscr A_{2n}\,\rvert\le 2^{\lceil n/2\rceil}\le 2^{(n+1)/2}.
$$
Hence
$$
\lvert \,\mathscr A_{\leq 2t}\,\rvert\leq \sum_{n=1}^{t}2^{(n+1)/2}
\;=\;
\sum_{n=1}^{t} {(\sqrt2)}^{\,n+1}$$ 
$$
\sum_{n=1}^{t} (\sqrt{2})^{\,n+1}
=
{2}\sum_{n=1}^{t} (\sqrt{2})^{\,n-1}
\
=
{2}\,\frac{(\sqrt{2})^{t}-1}{\sqrt{2}-1}
=
2(\sqrt{2}+1)\big(2^{t/2}-1\big).$$
That means $$
\lvert \,\mathscr A_{\leq2t}\,\rvert\le 2(\sqrt{2}+1)\big(2^{t/2}-1\big) .
$$
\\
On the other hand,
$$
\lvert \,\mathscr A_{\leq 2t} \,\rvert
\geq \sum_{\substack{1\leq n\leq t\\ n\ \mathrm{odd}}}\lvert\, \mathscr A_{2n}\, \rvert = \sum_{\substack{1\leq n\leq t\\ n\ \mathrm{odd}}}2^{\lceil n/2 \rceil}\geq2^{\lceil t/2\rceil}\geq 2^{t/2}.
$$

Therefore, from the above inequalities, we have 
\[\lvert \,\mathscr A_{\leq 2t} \,\rvert\simeq 2^{\frac{t}{2}}.\]

 \end{enumerate}

This completes the proof. \end{proof}

\begin{lemma}\label{3.4}
    $\hangindent=2em$
\begin{enumerate}
    \item $\lvert \;\mathscr{A}^{np}_{2t}\;\rvert\, \leq\, \frac{t}{2} 2^{\frac{t+2}{4}}$
    \item $\lvert \;\mathscr{A}^{np}_{\leq2t}\;\rvert\, \leq\, \frac{t^2}{2}2^{\frac{t+2}{4}}$
\end{enumerate}
\end{lemma}
\begin{proof}
    \begin{enumerate}
        \item  Using \cite[Lemma 2.4]{combigrowth}, we have
    $$ \lvert \;\mathscr{A}^{np}_{2t}\;\rvert\,=\sum_{s |t }\lvert \;\mathscr{A}^p_{2s}\;\rvert$$
    where the sum ranges over all proper divisors of $t$.   
$$\lvert \,\mathscr{A}^{np}_{2t}\,\rvert\,=\,\sum_{s |t }\lvert \,\mathscr{A}^p_{2s}\,\rvert\,\leq\, \sum_{s |t }\lvert \,\mathscr{A}_{2s}\,\rvert.$$
From \lemref{3.3}, $$\lvert \;\mathscr{A}^{np}_{2t}\;\rvert\,\leq \,\sum_{s |t }2^{\lceil s/2 \rceil} \,\leq\,\sum_{s |t }2^{\frac{s+1}{2}}. $$
Since the largest proper divisor of $t$ is $t/2$ and there are at most $t/2$ proper divisors, therefore
$$\lvert \;\mathscr{I}^{np}_{2t}\;\rvert\,\leq \,\frac{t}{2} 2^{\frac{t/2+1}{2}} = \frac{t}{2} 2^{\frac{t+2}{4}}\imp\lvert \;\mathscr{I}^{np}_{2t}\;\rvert\,\leq \,\frac{t}{2} 2^{\frac{t+2}{4}} . $$
Hence, $\lvert \;\mathscr{A}^{np}_{2t}\;\rvert\, \leq\, \frac{t}{2} 2^{\frac{t+2}{4}}.$
\item From above we have, 
   $$ \lvert \;\mathscr{A}^{np}_{\leq2t}\;\rvert\,=\sum_{n=1}^{t}\,\,\lvert\, \mathscr{A}_{2n}^{np}\,\rvert\,\leq\,\sum_{n=1}^{t}\,\frac{n}{2}2^{\frac{n+2}{4}}\,\leq\,\sum_{n=1}^{t}\,\frac{t}{2}2^{\frac{t+2}{4}}\,\leq\,\frac{t^2}{2}2^{\frac{t+2}{4}}.$$ \end{enumerate}

   Therefore, we have $\lvert \;\mathscr{A}^{np}_{\leq2t}\;\rvert\,\leq\,\frac{t^2}{2}2^{\frac{t+2}{4}}.$
\end{proof}
\begin{lemma}\label{3.5} As $t \to \infty$, 
   \begin{enumerate}
       \item  $ \lvert\,\mathscr{A}^p_{2t}\,\rvert\, \sim \,\lvert\,\mathscr{A}_{2t}\,\rvert$. 
       \item $ \lvert\,\mathscr{A}^p_{\leq2t}\,\rvert\, \sim \,\lvert\,\mathscr{A}_{\leq2t}\,\rvert$. 
       \end{enumerate}
\end{lemma}
\begin{proof}
    \begin{enumerate}
    \item    Using \lemref{3.4}, we have
    $$\lvert\,\mathscr{A}_{2t}\,\rvert - \frac{t}{2} 2^{\frac{t+2}{4}}\leq \lvert\,\mathscr{A}^p_{2t}\,\rvert \leq \lvert\,\mathscr{A}_{2t}\,\rvert$$ 
    Dividing by $\lvert\,\mathscr{A}_{2t}\,\rvert$ 
    $$1 - \frac{\frac{t}{2} 2^{\frac{t+2}{4}}}{\lvert\,\mathscr{A}_{2t}\,\rvert}\leq \frac{\lvert\,\mathscr{A}^p_{2t}\,\rvert}{\lvert\,\mathscr{A}_{2t}\,\rvert} \leq 1 $$
    observe from \lemref{3.3} that when $t$ is odd, $\lvert\,\mathscr{A}_{2t}\,\rvert \geq 2^{t/2}$, and when $t$ is even, $\lvert\,\mathscr{A}_{2t}\,\rvert=2^{\frac{t}{2}-1}$, therefore
     $$1 - \frac{\frac{t}{2} 2^{\frac{t+2}{4}}}{2^{t/2}}\leq \frac{\lvert\,\mathscr{A}^p_{2t}\,\rvert} {\lvert\,\mathscr{A}_{2t}\,\rvert} \leq 1 \text{ and } 1 - \frac{\frac{t}{2} 2^{\frac{t+2}{4}}}{2^{\frac{t}{2}-1}}\leq \frac{\lvert\,\mathscr{A}^p_{2t}\,\rvert} {\lvert\,\mathscr{A}_{2t}\,\rvert} \leq 1 $$
    $$1 - \frac{t}{2\cdot2^{\frac{t-2}{4}}}\leq \frac{\lvert\,\mathscr{A}^p_{2t}\,\rvert}{\lvert\,\mathscr{A}_{2t}\,\rvert} \leq 1 \text{ and } 1 - \frac{t}{2\cdot2^{\frac{t-6}{4}}}\leq \frac{\lvert\,\mathscr{A}^p_{2t}\,\rvert}{\lvert\,\mathscr{A}_{2t}\,\rvert} \leq 1$$
    Since $\dfrac{t}{2^{\frac{t-2}{4}+1}}$ and $\dfrac{t}{2^{\frac{t-6}{4}+1}}$ tends to $0$ as $t \to \infty$
    \\
    \\
    Hence,  $\lvert\,\mathscr{A}^p_{2t}\,\rvert\, \sim \lvert\,\mathscr{A}_{2t}\,\rvert.$
    \item For this part, we use the same arguments as above. We have from \lemref{3.4},
    $$\lvert\,\mathscr{A}_{\leq2t}\,\rvert - \frac{t^2}{2}2^{\frac{t+2}{4}}\leq \lvert\,\mathscr{A}^p_{\leq2t}\,\rvert \leq \lvert\,\mathscr{A}_{\leq2t}\,\rvert$$
     Dividing by $\lvert\,\mathscr{A}_{\leq2t}\,\rvert$ and observing that $\lvert \;\mathscr{A}_{\leq2t}\;\rvert\, \geq\,2^{t/2} $, it follows that
     $$1 - \dfrac{\frac{t^2}{2}2^{\frac{t+2}{4}}}{ 2^{t/2}}\leq \frac{\lvert\,\mathscr{A}^p_{\leq2t}\,\rvert}{\lvert\,\mathscr{A}_{\leq2t}\,\rvert} \leq 1 $$
     $$\imp 1 - \dfrac{t^2}{ 2^{\frac{t+2}{4}}}\leq \frac{\lvert\,\mathscr{A}^p_{2t}\,\rvert}{\lvert\,\mathscr{A}_{2t}\,\rvert} \leq 1.$$
     We know that $\dfrac{t^2}{ 2^{\frac{t+2}{4}}}\to 0$ as $t \to \infty.$ Hence,  $\lvert\,\mathscr{A}^p_{\leq2t}\,\rvert\, \sim \lvert\,\mathscr{A}_{\leq2t}\,\rvert.$
    \end{enumerate}
    This gives a proof. 
\end{proof}

\medskip From the results proved above, we can now give an exact count of ambiguous classes.
\subsection{Proof of \thmref{ambiguous}}
\begin{enumerate}
    \item By \lemref{3.3}, we have
\[
\lvert\,\mathscr A_{\le 2t}\,\rvert \simeq 2^{t/2}.
\]
In simple terms, as $t$ increases, the number of ambiguous classes of length at most $2t$ grows exponentially, with a growth rate comparable to $2^{t/2}$.
In other words, up to multiplicative constants, the growth rate of $\lvert\,\mathscr A_{\le 2t}\,\rvert$ is the same as that of $2^{t/2}$.

\item By \lemref{3.5}, we have
$\lvert\,\mathscr A^p_{\le 2t}\,\rvert \sim \lvert\,\mathscr A_{\le 2t}\,\rvert$ as  $t \to \infty$. 
This shows that the number of primitive ambiguous classes of length at most $2t$ is asymptotically equal
to the total number of ambiguous classes of length at most $2t$. \qed
\end{enumerate}
\bibliographystyle{plain}
	{\bibliography{ref}}

\end{document}